\documentclass[12pt, a4paper, reqno]{amsart}
\usepackage[english]{babel}
\usepackage{amsmath,amsfonts,amssymb,amsthm}
\usepackage{hyperref}
\usepackage{graphicx}
\usepackage{color}
\usepackage{caption}
\usepackage{subcaption}
\usepackage{longtable}

\usepackage[margin=2.5cm]{geometry}

\newtheorem{thm}{Theorem}[section]
\newtheorem{lem}[thm]{Lemma}

\newtheorem{con}[thm]{Conjecture}

\theoremstyle{definition}
\newtheorem{defi}[thm]{Definition}
\newtheorem{exam}[thm]{Example}
\theoremstyle{remark}
\newtheorem{rem}{Remark}

\newcommand { \ib }[1] {\textit{\textbf{#1}}}
\newcommand{\bd}{\stackrel{\rm bd}{\sim}}

\newcommand { \diag }{\mathop{\rm{diag}}\nolimits}

\newcommand{\dotcup}{\ensuremath{\mathaccent\cdot\cup}}

\DeclareMathOperator{\dens}{dens}
\DeclareMathOperator{\freq}{freq}

\newcommand{\R}{\mathbb{R}}
\newcommand{\Q}{\mathbb{Q}}
\newcommand{\Z}{\mathbb{Z}}
\newcommand{\C}{\mathbb{C}}
\newcommand{\X}{\mathbb{X}}

\begin{document}
\renewcommand{\ib}{\mathbf}
\renewcommand{\proofname}{Proof}
\renewcommand{\phi}{\varphi}
\renewcommand{\epsilon}{\varepsilon}
\newcommand{\conv}{\mathrm{conv}}

\title[Pisot substitutions, CPS and fractal BRS]{Pisot substitution sequences, one dimensional
cut-and-project sets and bounded remainder sets with fractal boundary}
\author{Dirk Frettl\"oh}
\address{Technische Fakult\"at, Bielefeld University}
\email{dirk.frettloeh@math.uni-bielefeld.de}

\author{Alexey~Garber}
\address{School of Mathematical \& Statistical Sciences, The University of Texas Rio Grande Valley, 1 West University Blvd., Brownsville, TX 78520, USA.}
\email{alexeygarber@gmail.com}


\date{\today}

\begin{abstract}
This paper uses a connection between bounded remainder sets in $\R^d$ and 
cut-and-project sets in $\R$ together with the fact that each one-dimensional 
Pisot substitution sequence is bounded distance equivalent to some lattice
in order to construct several bounded remainder sets with fractal boundary. 
Moreover it is shown that there are cut-and-project sets being not bounded 
distance equivalent to each other even if they are locally indistinguishable, 
more precisely: even if they are contained in the same hull.
\end{abstract}

\maketitle

\section{Introduction}
Number theorists study discrepancy sequences, which is currently a
very active field of research thanks to Tao and a polymath project \cite{tao}.
A lot of literature is dedicated to this subject. To mention just one example:
a special issue of \emph{Indagatioens Mathematicae} \cite{corput}
dedicated to Johannes Gualtherus van der Corput contains several results 
on discrepancy.

A particular part of this subject are bounded remainder sets (BRS). BRS have 
been studied intensely by several authors,
including early works by Hecke \cite{Hec}, Ostrowski \cite{Ost}, Sz\"usz 
\cite{Sz} and Kesten \cite{Kes}; see for instance \cite{GLev,HK} and 
references therein for an overview. BRS are sets $P \subset [0,1]^d$ with 
uniformly bounded discrepancy with respect to some
(discrete) toral rotation; see below for details. BRS have been studied
in the context of Diophantine approximation, dynamical systems or
quasi Monte Carlo methods.

Cut-and-project sets (CPS) and substitution sequences are (infinite)
discrete point sets in $\R$. They are central 
objects of study in the theory of aperiodic order, see \cite{BG} and
references therein. There is a strong connection between BRS and 
(windows of) CPS. This connection was implicit in several papers 
(see for instance \cite{DO2, HoZ}) but up to the knowledge of the 
authors the connection was explicitly spelled out only in \cite{H,HK}.

Here we use this connection together with several other results
in order to obtain many explicit examples of BRS with fractal boundary. 
Along the way we 
re-prove some of these results, namely the connection between
CPS and BRS (\cite{HK} respectively Theorem \ref{thm:brs-cps}), and 
the fact that the point sets generated by Pisot substitution 
are bounded distance equivalent to a lattice (\cite{HoZ}
respectively Theorem \ref{thm:pisot-bde}). We also use partial 
answers to the Pisot conjecture (Conjecture \ref{con:pisot}).

Independently we prove that two CPS are bounded distance equivalent
if their windows are equidecomposable by appropriate vectors
(Theorem \ref{thm:win-bde}), but that there are CPS with the
same data but shifted windows that are not bounded distance equivalent
to each other (Theorem \ref{thm:not-bde}). This implies that 
bounded distance equivalence of CPS is not preserved under local
indistinguishability, or more precisely: two CPS in the same
hull are not necessarily bounded distance equivalent to each other.

Because our results are dealing mainly with discrete point sets
in one dimension, most terminology is stated for this one-dimensional 
case only. Note that most objects mentioned in this paper 
(substitution sequences, tilings, CPS) have
higher dimensional analogues. Section \ref{sec:bde-cps} contains
some results for higher dimensions, namely, Lemma \ref{lem:hall},
Lemma \ref{lem:lattice_divide} and Theorem \ref{thm:win-bde}.

But mainly the objects of study in this paper are discrete point sets
on the line. One important example is the set $\Z$ of integers,
or more general: the set $a\Z = \{an \mid n \in \Z\}$ for some
$a \in \R \setminus \{0\}$. Such a set $a \Z$ is \emph{periodic},
that is, there is a non-zero $t \in \R$ such that $t+ a \Z = a\Z$ (namely,
all $t \in a\Z$ will do). We are more interested in studying \emph{nonperiodic}
discrete point sets $\Lambda \subset \R$, that is, sets $\Lambda$
such that $\Lambda + t = \Lambda$ implies $t=0$. 
Two interesting classes of nonperiodic discrete point sets are
substitution sequences and CPS. They are 
introduced in Section \ref{sec:subst-cps}.

The main question motivating this paper is ``when are two given discrete
point sets on the line bounded distance equivalent to each other?''
Roughly speaking, this means that there is a perfect matching
between the point sets such that the distance between any two 
matched points is uniformly bounded. More precisely:
\begin{defi}
Two discrete point sets $\Lambda, \Lambda' \in \R^d$ are called 
\emph{bounded distance equivalent} ($\Lambda \bd \Lambda'$) if there 
is some $c>0$ and some bijection $\phi \colon  \Lambda \to \Lambda'$ 
such that $|x - \phi(x)| < c$ for all $x \in \Lambda$. 
\end{defi}
\begin{rem} It is straight-forward to see that bounded distance 
equivalence is indeed an equivalence relation. The minmal $c$ can even
be regarded as the distance of two Delone sets $\Lambda \bd \Lambda'$,
which can easily be checked to define a metric on each equivalence class.
\end{rem}
A cut-and-project set (CPS) is a discrete point set on the line that is
defined by projecting the points of some point lattice in $\R^{d+1}$
that are contained in some strip $W \times \R$ to the line $\R$, where 
$W$ is some (usually compact)  subset of $\R^d$. Compare Figure 
\ref{fig:cps-fib} for an example with $d=1$, see Section \ref{sec:subst-cps}
for a precise explanation. It turns out that the question whether a 
given (nonperiodic) CPS is bounded distance equivalent 
to some (periodic) point lattice is strongly related to the question 
whether $W$ is a bounded remainder set. 
\begin{rem}
Throughout the paper $\mu$ denotes Lebesgue measure in the appropriate dimension.
\end{rem}
\begin{defi}
Let $\alpha \in \R^d$, $W \subset [0,1]^d$. $W$ is a \emph{bounded remainder set}
(BRS) with respect to $\alpha$, if for the \emph{discrepancy} 
\[ D_n(W,x) = \sum_{k=0}^{n-1} 1_W\bigl((x+k \alpha) \bmod 1\bigr) \; - n \mu(W) \]
holds: there is $c>0$ such that for all $n$ and for almost all $x \in [0,1[$ holds
$|D_n(W,x)| < c$. 
\end{defi} 
The parameter $\alpha$ in the definition above is some irrational number
(if $d=1$) respectively irrational vector (if $d>1$) which in our
setup depends on the point lattice defining the cut-an-project set under 
consideration. The $x$ in the definition above is negligible for
our purposes.

\section{Substitutions and cut-and-project sets} \label{sec:subst-cps}

A \emph{point lattice} (or shortly \emph{lattice}) in $\R^d$ is 
the integer span of $d$ linearly independent vectors. Hence
a lattice in $\R$ is a set $a \Z = \{n a \, | \, n \in \Z \}$, where 
$a \in \R\setminus \{0\}$.

The \emph{density} of a discrete point set $\Lambda$ in $\R$ is the
average number of points of $\Lambda$ per unit. It is defined 
as the limit $\lim\limits_{n \to \infty} \frac{1}{2n} | \Lambda \cap [-n,n]|$,
if it exists. (In fact, the definition of the density of a discrete point set
is much more subtle, compare for instance \cite[Section 1, p 16]{BG}. For our 
purposes this definition, sometimes called \emph{central density}, is sufficient.)

The density of a lattice in $\R^d$ is $\frac{1}{|\det(M)|}$, where $M$ is 
the matrix whose columns are the spanning vectors of the lattice. Hence 
the density of a lattice $L=\{ na \, | \, n \in \Z \}$ is $\frac{1}{|a|}$.

Let $\mathcal{A}_n = \{a_1, \ldots, a_n\}$ denote an alphabet of 
$n$ letters. Let $\mathcal{A}^{\ast}_n$ denote all finite words over 
$\mathcal{A}_n$.
A symbolic substitution is a map $\sigma: \mathcal{A}^{\ast}_n \to
\mathcal{A}^{\ast}_n$ such that for all $u,v \in
\mathcal{A}^{\ast}_n$ holds $\sigma(u)\sigma(v) = \sigma(uv)$. 
Here $uv$ denotes the concatenation of $u$ and $v$. Hence
a symbolic substitution is uniquely defined by $\sigma(a_1),
\ldots, \sigma(a_n)$. The \emph{substitution matrix} of
a substitution $\sigma$ is the matrix $M_{\sigma} =
(|\sigma(a_j)|_i)_{1 \le i,j \le n}$, where $|\sigma(a_j)|_i$
denotes the number of letters $a_i$ in $\sigma(a_j)$.
A substitution $\sigma$ is called \emph{primitive}, if there
is $k$ such that $M_{\sigma}^k$ contains positive entries only. 
By the Perron-Frobenius theorem \cite{P}, the matrix $M_{\sigma}$ of a 
primitive substitution $\sigma$ has
a unique real eigenvalue $\lambda>0$, such that all
other eigenvalues of $M_{\sigma}$ are strictly smaller in absolute value.
Furthermore, $M_{\sigma}$ has an eigenvector $v$ for $\lambda$ 
with positive entries only. In the following, we refer to 
$\lambda$ as Perron-Frobenius eigenvalue of $M_{\sigma}$
and to $v$ as Perron-Frobenius eigenvector of $M_{\sigma}$.
The \emph{symbolic hull} of $\sigma$ is
\[ X_{\sigma} := \{ u \in \mathcal{A}_n^{\Z} \, | \, \mbox{each finite subword 
of }u\mbox{ is contained in some } \sigma^k(a_i) \}. \]
\begin{rem} In general, the hull $X_u$ of any symbolic sequence 
$u = (\ldots, u_{-1}, u_0, u_1,\ldots ) \in \mathcal{A}^{\Z}$ is the closure 
of $\{ S^k u \mid k \in \Z \}$ with respect to the metric 
\[ d(u,u):=0, \quad d(u,v) := 2^{-\min\{ |i| \mid u_i \ne v_i \}} \quad \mbox{ where }
v = ( \ldots , v_{-1}, v_0, v_1, \ldots ) \ne u \]  
Here $S$ denotes the shift operator $S (\ldots, u_{-1}, u_0, u_1,\ldots )
= (\ldots, v_{-1}, v_0, v_1, \ldots)$ with $v_i = u_{i+1}$. 
If $\sigma$ is a primitive substitution then $X_{\sigma} = X_u$
for all $u \in X_{\sigma}$ \cite[Remark 4.2]{BG}. 
\end{rem}

A tile substitution in $\R$ is a collection of intervals $[0,\ell_1],
\ldots, [0,\ell_n]$, an \emph{inflation factor} $\lambda$, and 
a rule how to partition each $\lambda [0, \ell_i]$ into translates
of the original intervals $[0,\ell_1], \ldots, [0,\ell_n]$. Translates of
the $[0,\ell_1], \ldots, [0,\ell_n]$ are called \emph{tiles}. 
A tile substitution can be repeatedly applied to a single tile, 
covering larger and larger portions of the line. The iterate 
$\sigma^k([0,\ell_i])$ is called \emph{(level-$k$) supertile} (for 
convenience, we will sometimes use $a_i$ for the tile $[0,\ell_i]$
below).

Trivially, any tile substitution can be transformed into a symbolic substitution by
identifying translates of intervals $[0, \ell_i]$ with letters $a_i$.
Because of the following folklore result (see e.g. \cite[Section 4.1]{BG})
each primitive symbolic substitution can be turned into a tile
substitution, too. 
\begin{thm} \label{thm:natural}
Let $\sigma$ be a primitive symbolic substitution. Let $\lambda$
be the Perron-Frobenius eigenvalue of $M_{\sigma}$,
and let $w=(\ell_1, \ldots, \ell_n)$ be a left 
eigenvector of $M$ for $\lambda$. Identifying $a_i$ with
$[0, \ell_i]$ yields a tile substitution $\sigma$ with inflation
factor $\lambda$. 
\end{thm}
The $\ell_i$ in the result above are called \emph{natural
tile lengths}. The choice is unique up to an overall scaling
of the $\ell_i$. Any symbolic infinite sequence in the hull
of some symbolic substitution can be turned into a tiling of the line
by identifying the symbol $a_i$ with a translate of the
tile $[0,\ell_i]$. For the sake of convenience we want to identify
the symbol $u_0=a_i$ in $u = \cdots u_{-1} u_0 u_1 \cdots$ with
the tile $[0, \ell_i]$ with left endpoint 0. 
Since we will consider discrete point sets in the sequel, we
consider the point sets $\{ \ldots, x_{-1}, x_0=0, x_1, \ldots \}$
consisting of the left endpoints of the tiles in the tiling rather than
the tilings itself. The set of all point sets arising from some 
primitive substitution (symbolic or tile) in this way is the \emph{geometric
hull} and is denoted by $\X_{\sigma}$. Theorem \ref{thm:natural} has the following 
consequences for these point sets $\Lambda \in \X_{\sigma}$ \cite[Section 4.4]{BG}. 

\begin{thm} \label{thm:eigv-dens}
Let $\sigma$ be a primitive substitution, and let $u \in \X_{\sigma}$.
Let $\lambda$ be the Perron-Frobenius eigenvalue of $M_{\sigma}$,
and let $v=(v_1, \ldots, v_n)^T$ the right normalized eigenvector of
$M$ for $\lambda$. The relative frequency of symbol $a_i$ in $u$ 
\[ \freq(a_i) = \lim_{n \to \infty} \frac{ |\{ u_k \, | \, u_k=a_i, \, -n < k \le n \}|}
{2n} \]
exists and equals $v_i$. The density of each point set $\Lambda = 
\{ \ldots, x_{-1}, x_0=0, x_1, \ldots \} \in \X_{\sigma}$ using natural 
tile lengths $w=(\ell_1, \ldots, \ell_n)$ exists and equals 
\[ \dens(\Lambda) =  \Big( \sum_{i=1}^n v_i \ell_i \Big)^{-1} = 
\frac{1}{w \cdot v}. \]
\end{thm}

\begin{rem}
By the normalized vector $v$ we always mean that the sum of its 
coordinates equals $1$. In particular the theorem implies that the sum of 
relative frequencies of all symbols is $1$, as it ought to be.
\end{rem}

\begin{exam} \label{ex:fib1}
The Fibonacci substitution $\sigma$ is given by $a \mapsto ab$,
$b \mapsto a$. It produces a bi-infinite sequence by repeatedly applying
it to the pair of letters $a|a$: $\sigma(a|a)=ab|ab$, $\sigma^2_f(a|a)=
aba|aba$, $\sigma^3_f(a|a)=abaab|abaab$, $\sigma^4_f(a|a)=
abaababa|abaababa$,  
$\ldots$. Its substitution matrix is $M_{\sigma} = \big( \begin{smallmatrix}
1 & 1\\ 1 & 0 \end{smallmatrix} \big)$ with Perron-Frobenius eigenvalue 
$\lambda=\frac{1}{2}(1+\sqrt{5})=:\tau$ and second eigenvalue
$\frac{1}{2}(1-\sqrt{5})$. The normalized eigenvector of $M_{\sigma}$
for $\lambda$ is $(\frac{\tau}{1+\tau},\frac{1}{1+\tau})^T$,
hence the relative frequency of $a$ is $\frac{\tau}{1+\tau}$ and
the relative frequency of $b$ is $\frac{1}{1+\tau}$.

The sequence can be transformed into a discrete point set in 
$\R$ by assigning an interval of length $\tau$ to $a$
and an interval of length 1 to $b$. We then consider the endpoints
of the intervals as our point set $\Lambda_F$. Then the density
of $\Lambda_F$ is $\Big( ( \tau, 1) \cdot (\frac{\tau}{1+\tau},\frac{1}{1+\tau})^T
\Big)^{-1} = \frac{1+\tau}{2+\tau}$.
\end{exam}
\begin{defi} \label{def:cps}
A \emph{cut-and-project set} (CPS, aka \emph{model set}) 
$\Lambda$ is given by a collection of maps and spaces:
\[ \begin{array}{ccccc}
G=\R & \stackrel{\pi_1}{\longleftarrow} & \R \times \R^d & 
\stackrel{\pi_2}{\longrightarrow} & H=\R^d\\
\cup & & \cup & & \cup\\
\Lambda & & \Gamma & & W
\end{array} \]
where $G$ and $H$ are locally compact abelian groups, 
$\Gamma$ is a lattice (i.e., a discrete cocompact subgroup)
in $G \times H$, $W$ is a relatively compact 
set in $H$, and $\pi_1$ and $\pi_2$ are projections to $G$, 
respectively to $H$, such that 
$\pi_1|_\Gamma$ is one-to-one, and $\pi_2(\Gamma)$ is
dense in $W$. Then 
\[ \Lambda = \{ \pi_1(x) \, | \, x \in \Gamma, \, \pi_2(x) \in W \} \]
is a CPS. 
\end{defi}
\begin{rem}
For our purposes it does not really matter whether if we replace
``lattice'' by ``translate of a lattice'' in the definition above: 
translating the lattice is equivalent to translating the window
--- and the CPS $\Lambda$, if needed --- by some small amount.
\end{rem}
Throughout the paper $G$ usually equals $\R$,
and $H$ always equals $\R^d$ for some $d \ge 1$. The following 
result is standard, see for instance \cite[Section 7.1, 7.2]{BG}. 

\begin{lem} \label{lem:denscps}
The density of a CPS $\Lambda$ is $\frac{\mu(W)}{|\det(M_{\Gamma})|}$,
where $M_{\Gamma}$ is the matrix whose columns are the spanning 
vectors of the lattice $\Gamma$. 
\end{lem}

\begin{figure}
\includegraphics[width=.9\textwidth]{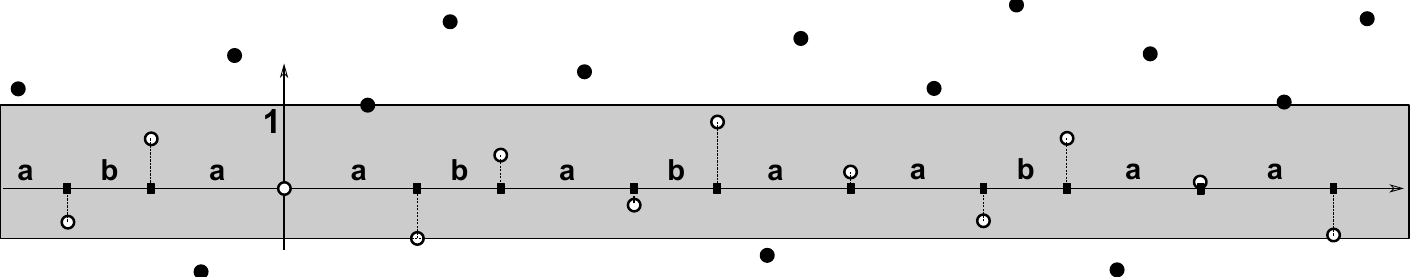} 
\caption{An illustration of the cut-and-project scheme for the
Fibonacci sequence. The symbolic Fibonacci sequence is indicated
by letters a and b. The vertex set of the resulting tiling are the
black marks on the $x$-axis. They are projections of all the
lattice points contained in the gray strip (white round points).
The gray strip is $\R \times [\frac{-1}{\tau},1[$. \label{fig:cps-fib}}
\end{figure}
\textbf{Notation:} In the sequel we will take advantage of the following
convention: Similar to a point set arising from a substitution, a CPS 
$\Lambda \subset \R$ can be written as an increasing sequence $\{ \ldots, x_{-1},
x_0=0, x_1, x_2, \ldots \}$ with $x_i<x_j$ for $i<j$, and where 0 is labelled
$x_0$. 

\begin{exam} \label{ex:fib2} (Example \ref{ex:fib1} continued)
A Fibonacci point set $\Lambda_F$, as constructed by the substitution above,
can as well be generated as a CPS. This CPS
has $d=1$, $W=[-\frac{1}{\tau}, 1[$, lattice $\Gamma = \langle \big(
\begin{smallmatrix} 1\\ 1 \end{smallmatrix} \big), \big( 
\begin{smallmatrix} \tau \\ \frac{-1}{\tau} \end{smallmatrix} \big) \rangle_{\Z}$, 
and $\pi_1$ and $\pi_2$ are orthogonal projections; compare Figure 
\ref{fig:cps-fib}. Consequently the density of $\Lambda_F$ is 
$\frac{1+\frac{1}{\tau}}{\tau + \frac{1}{\tau}} = \frac{\tau+1}{\tau+ 2}$.
\end{exam}

\section{Substitution sequences in bounded distance to a lattice}

\begin{defi}
Let $\mathcal{A}=\{a_1,\ldots, a_m\}$ and let $\sigma$ be a primitive symbolic 
substitution over $\mathcal{A}$. The substitution is called a 
{\it Pisot substitution} 
if the matrix $M_\sigma$ of the substitution has one (simple) positive 
real eigenvalue $\lambda>1$, and all other eigenvalues (real or complex) 
have absolute values less than $1$.
\end{defi}
Note that sometimes a Pisot substitution is defined to obey a stronger condition,
namely, that in addition to the definition above all eigenvalues 
different from $\lambda$ are non-zero, see for instance 
\cite[1.2.5]{Fogg}. Since we are dealing with integer matrices this condition
implies the characteristic polynomial of $M_{\sigma}$ being irreducible.
Theorem \ref{thm:pisot-bde} below holds for the more
general case when some eigenvalues are zero, hence we stick to the more general
definition here.

The goal of this section is to give an elementary proof that the set of vertices of a one-dimensional primitive Pisot substitution is bounded distance equivalent to some lattice. This was done before, for instance by Holton and Zamboni \cite[Theorem A.2]{HoZ}. (Probably this result was already known to Rauzy, but we are not aware of any reference). A similar result holds for Pisot substitution tilings in higher dimension, shown by Solomon \cite{Sol}. The proof here uses much simpler techniques than the ones by Holton and Zamboni, and by Solomon. Although we need to require additionally that the substitution matrix is diagonalizable, we gain that our result also includes the case of eigenvalue(s) zero.

We proceed to show that any $\Lambda' \in \X_{\sigma}$ is bounded distance
equivalent to $\dens^{-1}(\Lambda') \Z$ by showing this for some fixed point
of $\sigma$. That means we choose $\Lambda$ such that for the biinfinite
symbolic sequence $u$ corresponding to $\Lambda$ holds $\sigma(u)=u$.
Such a $u$ can always be found by the pigeonhole principle, possibly by 
using some power $\sigma^k$ of $\sigma$ instead of $\sigma$ itself. 
Using the following result it suffices to show the claim for $\Lambda$. 
Since it holds in higher dimensions as well we state it for the general case.

\begin{thm}\label{thm:hull-lattice}
Let $\Lambda$ be a discrete point set in $\R^d$ generated by some
primitive substitution, or some CPS, such that $\Lambda \bd a \Z^d$ 
for some $a$. Then for any $\Lambda'$ in the hull of $\Lambda$ holds:
$\Lambda' \bd a \Z^d$.
\end{thm}
\begin{proof}
We give only a sketch of the proof based on the criterion of Laczkovich 
\cite{Lac}. In plain words, this criterion says that given a set $\Lambda \subset \R^d$ 
then $\Lambda \bd a \Z^d$ if and only if the number of points of $\Lambda$ in any union of 
parallel unit cubes differs from the $d$-dimensional volume of that union by some constant 
(which depends on $M$ only) times the $(d-1)$-dimensional volume of the surface of 
that union. Spelled out it
says that for $d>1$ in $\R^d$ holds: $\Lambda \bd a^{-1/d} \Z^d$ if and only if
there is a positive constant $C>0$  such that $\big| |\Lambda \cap H| - 
a \lambda_d(H)\big| \le C \lambda_{d-1}(\partial H)$ holds whenever 
$H$ is a finite union of unit cubes.

Since $\Lambda \bd a \Z^d$ this condition holds for $\Lambda$. 
Since $\Lambda$ and $\Lambda'$ are in the same hull, any pattern in 
$\Lambda'$ can be found in $\Lambda$ and therefore the condition of 
Laczkovich holds for $\Lambda'$ as well.
\end{proof}

The argument in the sequel goes along the following lines: 
First we bound the distance between the endpoint of a $k$-level
supertile in $\Lambda$ and the corresponding point in 
$\dens^{-1}(\Lambda) \Z$ by some term $c \cdot c_{\sigma}^k$ 
with $c_{\sigma}<1$. Then we show that the error between any point in 
$\Lambda$ and its corresponding point in $\dens^{-1}(\Lambda) \Z$ 
is bounded by $C (c_{\sigma}^0 +  c_{\sigma}^1 + c_{\sigma}^2 + \cdots) + C'$
for some constants $C, C'$. Hence the distance is uniformly bounded. 

Let us consider some level-$k$ supertile $\sigma^k(a_i)$ of type $i$
with left endpoint 0. It consists of $n^i_k$ tiles, say, i.e. of
$n^i_k$ intervals, each having length in $\{\ell_1, \ldots, \ell_m\}$,
where $\{\ell_1, \ldots, \ell_m)$ is the left eigenvector of the substitution
matrix $M_{\sigma}$ for $\lambda$, as in Theorem \ref{thm:natural}.
In the sequel let $x^i_j$ denote the $j$th point in the set of endpoints
of these intervals. So the point set representing this supertile is 
$\{x^i_0=0, x^i_1, \ldots, x^i_{n^i_k} \}$. This supertile consists of 
$n^i_k=(1,\ldots,1)M_\sigma^k\mathbf{e}_i$ intervals of length 
$\ell_1,\ldots,\ell_m$ where $\mathbf{e}_i$ is the $i$th vector of the 
standard basis in $\mathbb{R}^m$. By Theorem \ref{thm:eigv-dens} for any 
$i \in \{1, \ldots, m\}$ 
the average length of intervals in the sequence $\{x_j^i\}_j$ equals 
the scalar product $\dens^{-1}(\Lambda)$ of the left and right 
$\lambda$-eigenvectors of the matrix $M_\sigma$ (the right eigenvector 
normalized). Hence we have by the definition of the density that the
average tile length approaches $\dens^{-1}(\Lambda)$. That is,
\begin{equation} \label{eq:lim-dens}
\lim\limits_{j \to \infty}\left|\frac{x^i_j}{j}-
\dens^{-1}(\Lambda)\right|=0.
\end{equation}
In the sequel we need only the fact that $\lim\limits_{k \to \infty}\left|
\frac{x^i_{n^i_k}}{n^i_k}-\dens^{-1}(\Lambda)\right|=0$.


\begin{lem}\label{lem:supertileerror}
Let the substitution matrix $M_{\sigma} \in \R^{m \times m}$ be diagonalizable
over $\C$. There is a number $c_\sigma\in[0,1[$ such that for every 
$i\in\{1, \ldots, m\}$ there is some positive constant $c_i$ such that
\[\left|x^i_{n^i_k}-n^i_k\cdot \dens^{-1}(\Lambda)\right|<c_ic_\sigma^k.\]
\end{lem}
\begin{proof}
Let's fix $i$. We will write down the left hand side using the 
substitution matrix. The number $x^i_{n^i_k}$ is the length of level-$k$ 
supertile of type $i$. After every substitution step the length of each 
basic tile multiplies by $\lambda$, thus $x^i_{n^i_k}=\lambda^k \ell_i$. 

We know that $n^i_k=(1,\ldots,1)M_\sigma^k\mathbf{e}_i$. Since the matrix 
$M_\sigma$ is diagonalizable, there is a matrix $C$ (possibly with complex 
entries) such that $M_\sigma=C\diag(\lambda,\lambda_2,\ldots,\lambda_m)C^{-1}$ 
where $\lambda_i$ are all algebraic conjugates of $\lambda$. Then 
\[M^k_\sigma=C\diag(\lambda^k,\lambda^k_2,\ldots,\lambda^k_m)C^{-1}=\lambda^k C\diag(1,0,\ldots,0)C^{-1}+B_k\]
for some matrix $B_k$ with entries represented by linear forms of 
$\lambda_2^k,\ldots,\lambda_m^k$.

Plugging this back in the formula for $n^i_k$ we get
\[x^i_{n^i_k}-n^i_k\cdot \dens^{-1}(\Lambda)=b \cdot \lambda^k+(1,\ldots,1)B_k\mathbf{e}_i \cdot \dens^{-1}(\Lambda)\]
for some number $b$. Note, that $\lambda$ is the only eigenvalue of $M_\sigma$ with absolute value greater than $1$, hence $\lim\limits_{k\rightarrow\infty}B_k=0$ and, 
using Equation \eqref{eq:lim-dens}:
 \[0=\lim\limits_{k\to \infty}\left|\frac{x^i_{n^i_k}}{n^i_k}-\dens^{-1}(\Lambda)\right|=
 b \cdot \lim_{k \to \infty} \frac{\lambda^k}{n^i_k} = b \cdot \dens^{-1}(\Lambda).\]
Hence $b=0$, and the difference $\left|x^i_{n^i_k}-n^i_k\cdot \dens^{-1}(\Lambda)\right|$ can be written as a fixed linear form of $k$th powers of algebraic conjugates of $\lambda$ which have absolute values less than $1$.

Now we can choose $c_\sigma=\max(|\lambda_2|,\ldots,|\lambda_m|)<1$ and the statement of the lemma is obvious.
\end{proof}

Let $\epsilon^i_k$ be the maximum discrepancy (or error) within
 the level-$k$ supertile of type $i$, that is 
\[\epsilon^i_k=\max\limits_{0<j\leq n^i_k}|x^i_j-j\cdot \dens^{-1}(\Lambda)|.\]
Our next step is to establish recurrent inequalities for the errors $\epsilon^i_k$.

\begin{lem}\label{lem:error}
For every $k>1$ we have 
\[\epsilon^i_k\leq \max\limits_{1\leq j\leq m}\epsilon^j_{k-1}+C\cdot c_\sigma^{k-1}\]
where $C$ is a constant that does not depend on $k$ or $i$.
\end{lem}
\begin{proof}
Let $N$ be the maximum number of tiles in the level-1 supertile, and let $C_{max}$ be the maximum of constants $c_j$ from Lemma \ref{lem:supertileerror}. Then we can take $C=C_{max}N$.

Indeed, assume that the maximum for $\epsilon^i_k$ is attained at the 
point $x^i_t$. Then this point falls into one of $n^i_1$ supertiles of level $k-1$ 
that constitute the level-$k$ supertile. Assume it falls in the $s$th supertile 
from left. Then according to Lemma \ref{lem:supertileerror} each of the first 
$s-1$ level $k-1$ supertiles contributes an error of at most $C_{max}c_\sigma^{k-1}$, and the 
error in the last supertile is at most $\max\limits_{1\leq j\leq m}\epsilon^j_{k-1}$. 
Since $s\leq n^i_1\leq N$ we obtain
\[\epsilon^i_k\leq (s-1)C_{max}c_\sigma^{k-1}+\max\limits_{1\leq j\leq m}\epsilon^j_{k-1}\leq \max\limits_{1\leq j\leq m}\epsilon^j_{k-1}+C\cdot c_\sigma^{k-1}. \qedhere\]
\end{proof}

\begin{thm} \label{thm:pisot-bde}
Let $\sigma$ be a primitive Pisot substitution in $\R$ such that the
substitution matrix $M_{\sigma}$ is diagonalizable over $C$. For each 
$\Lambda \in \X_{\sigma}$ holds $\Lambda \bd \dens^{-1}(\Lambda) \Z$. 
\end{thm}
\begin{proof}
By Theorem \ref{thm:hull-lattice} it is sufficient to show the claim for 
some $\Lambda$ arising from a fixed point of $\sigma$. Since $\Lambda$ is 
a nested sequence of supertiles it is sufficient to show that all sequences 
$\epsilon^i_k$ are uniformly bounded. Lemma \ref{lem:error} implies 
$\epsilon^i_k\leq \max\limits_{1\leq j\leq m}\epsilon^j_{k-1}+C\cdot c_\sigma^{k-1}$ 
and therefore 
\[\max\limits_{1\leq j\leq m}\epsilon^j_{k}\leq \max\limits_{1\leq j\leq m}\epsilon^j_{k-1}+C\cdot c_\sigma^{k-1}.\]
Applying this inequality to the errors $\epsilon^j_{k-1},\epsilon^j_{k-2}$ and 
so on, we obtain that
\[\max\limits_{1\leq j\leq m}\epsilon^j_{k}\leq C\cdot c_\sigma^{k-1}+C\cdot c_\sigma^{k-2}+\ldots+C\cdot c_\sigma^{0}+\max\limits_{1\leq j\leq m}\epsilon^j_{0}.\]
So, each error $\epsilon^i_k$ does not exceed the sum of the infinite 
geometric series $\sum\limits_{i=0}^\infty Cc_\sigma^i$ with positive quotient 
$c_{\sigma}<1$ plus some constant. Thus all errors are bounded.
\end{proof}

\begin{rem} \label{rem:aabbab}
As mentioned above, some sources do not allow $0$ to be an eigenvalue of 
the matrix of a Pisot substitution. But our approach works in that case, too.
For instance, consider the substitution $a \mapsto aabb, b \mapsto ab$.
The substitution matrix is $\big( \begin{smallmatrix} 2 & 1\\ 2 & 1
\end{smallmatrix} \big)$ with eigenvalues 3 and 0. The natural tile
lengths are $\ell_1=2$, $\ell_2=1$, the frequencies are $\frac{1}{2}$
for both $a$ and $b$. The discrete point set $\Lambda$ arising from 
this substitution is nonperiodic and has density $\frac{2}{3}$. 
Theorem \ref{thm:pisot-bde} implies that $\Lambda \bd \frac{3}{2} \Z$.

Additionally, the same arguments will work not only if $M_\sigma$ is 
diagonalizable but also if some power of $M_\sigma$ is diagonalizable. 
The latter condition means that for any non-zero eigenvalue of $M_\sigma$ 
all Jordan cells have size $1$ but Jordan cells corresponding to the 
eigenvalue $0$ can be arbitrary.
\end{rem}

\section{BRS and bounded distance equivalence of CPS} \label{sec:cps-brs}

In the sequel we will exploit a correspondence between CPS and
BRS in $\R^d$. A relevant result in this context is the following theorem by Kesten
that gives a necessary and sufficient condition for an interval $[a,b] \subset
[0,1[$ to be a BRS with respect to some given slope $\alpha$. 

\begin{thm}[\cite{Kes}] \label{thm:kes}
Let $\alpha \in [0,1[$, $0 \le a < b \le 1$. Then $[a,b]$ 
is a BRS with respect to $\alpha$ if and only if $b-a \in \alpha \Z + \Z$.
\end{thm}

The ``if'' part of Theorem \ref{thm:kes} was known for a long time. 
It was shown by Hecke
\cite{Hec} (p 73) for $a=0$, using analytic number theory
(Dirichlet series of meromorphic functions). Ostrowski found a simple
argument for generalizing this to arbitrary $a$ \cite{Ost}. Kesten
settled the ``only if'' part by ``heavy use of continued fraction
expansions''. A simpler proof in a more general context was given by
Furstenberg, Keynes and Shapiro \cite{FKS}. The result was generalized to
sets that are not necessarily intervals by Oren \cite{oren}. The results
of Kesten and Oren have been proved by topological methods in \cite{KS}.
Several papers studied the higher dimensional analogues
of the problem, see \cite{GLev,HKK} for an overview.

The ``if''-part of Theorem \ref{thm:kes} can be proven by a simple
geometric argument \cite{DO1, HKK}. Here we want to give a brief sketch
of the idea on a very concrete level, since it illustrates nicely the
connection between BRS and CPS that we will make more precise below. 

Let $a=0$, and consider the set 
\[ \Lambda = \{ k \in \Z \, | \, k \alpha \bmod 1 \in [0,b[ \}. \]
We lift $\Lambda$ to $\R^2$ as follows:
Let $\Gamma:= \langle (1, \alpha)^T, (1,\alpha-1)^T \rangle_{\Z}$. 
In particular, $\Gamma$ contains all points $(k, k \alpha \bmod 1)^T$
where $k \in \Z$. 
If $\pi_1$ denotes the orthogonal projection to the first coordinate,
then $\Lambda$ consists of all elements $\pi_1(k, k \alpha \bmod 1)$ of $\Gamma$ 
with $0 \le k \alpha \mod 1 < b$; compare Figure \ref{fig:mod1-2}, left
part: the points in the gray strip are the points that are projected
to points in $\Lambda$. (Comparing this left part of the image with 
Figure \ref{fig:cps-fib} 
already gives an idea about the connection between CPS and BRS.)

The condition $b \in \Z + \alpha \Z$ means that there is
some lattice point $(k, k \alpha \bmod 1) \in \Gamma$ such that 
$b = k \alpha \bmod 1$. Rather than projecting all points orthogonally
to the line (by $\pi_1$) we project now in direction of 
$(k, k \alpha \bmod 1)$. With respect to the points in the strip
this is clearly a bounded distance transformation.

Fixing $k \in \Z \setminus \{0\}$ and attaching to each lattice point in 
$\Gamma$ a line segment $(k, k \alpha \bmod 1)$ (the ``flag-sticks'' in the image)
yields a set of parallel equidistant lines. Because of the particular 
property of $b$ being equal to $k \alpha \bmod 1$ for some appropriate 
$k \in \Z$, the line segments are in fact a partition of the lines (or 
more general, an $m$-fold covering of the lines). 
Projecting all points in some strip in direction $(k, k \alpha \bmod 1)$ 
then is the same as considering the intersection of the parallel lines 
with the $x$-axis. 

The intersection of the parallel equidistant lines with the $x$-axis is a 
periodic set $c \Z$. In fact, it is almost a CPS: it fulfils all
requirements but the one that the for projection $\pi$ on $\R$ holds
that $\pi: \Gamma \to \R$ is one-to-one. In fact this can be fixed easily
by changing the direction of projection slightly. Hence, by Lemma
\ref{lem:denscps} the density of the resulting (almost) CPS $c \Z$ 
is $\frac{b}{|\det(M_\Gamma)|} = b$ 
(since $\Gamma = \langle \big( \begin{smallmatrix} 1 \\ \alpha 
\end{smallmatrix} \big),  \big( \begin{smallmatrix} 1 \\ \alpha -1 
\end{smallmatrix} \big) \rangle_{\Z}$). Thus we obtain $c \Z = \frac{1}{b} \Z$.  
Altogether we have $\frac{1}{b} \Z \bd \Lambda$. The if-part of Kesten's 
theorem follows now like in the proof of Theorem \ref{thm:brs-cps} below. 

\begin{figure}
\includegraphics[width=.98\textwidth]{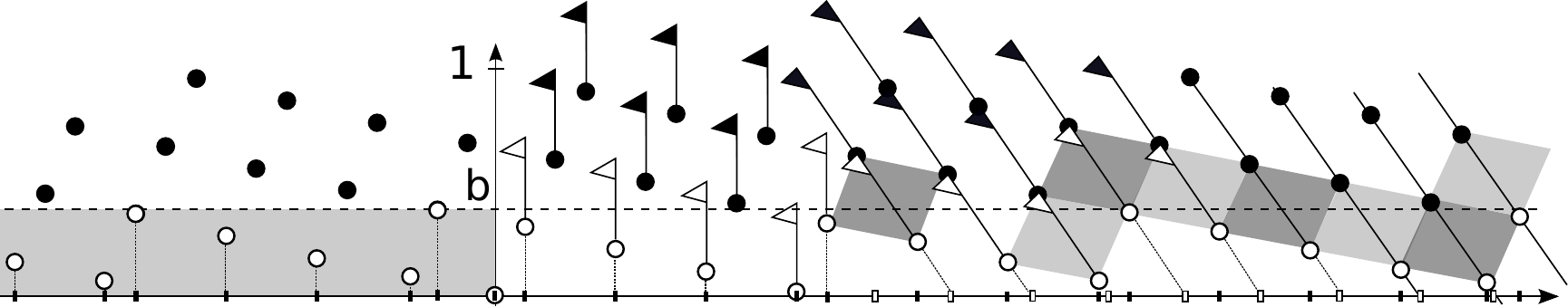}
\caption{A graphical sketch of the connection between CPS and BRS.
Projecting all points in 
the strip to the $x$-axis (left part) is the same as assigning to 
each point a flag-stick of height $b$, and projecting each point whose 
flag-stick is hit by the horizontal line through $(0,b)$ (middle).
By slanting the flag-sticks and projecting not orthogonally, but in
direction of the flag-sticks, the resulting set of projected points
becomes a periodic set (white rectangles, right).  \label{fig:mod1-2}}
\end{figure}

The argument above makes use of a relation between BRS and CPS: the
set $\Gamma$ above is the lattice in $\R \times \R$, the set $\Lambda$
corresponds to the CPS. What is lacking is the requirement that 
the projection $\pi_1|_{\Lambda}$ is one-to-one. This can be achieved
easily by changing the direction of projection, or equivalently, by applying  
a linear transformation $\big( \begin{smallmatrix} 1 & a\\ 0 & 1 
\end{smallmatrix} \big)$ to $\Gamma$.
In the sequel we want to translate an arbitrary CPS into the
canonical form used in the reasoning above, i.e. into a set 
that fulfils all the conditions of a CPS but one,
having lattice $\Gamma' = \langle v_1, \ldots, v_{d+1} \rangle_{\Z}$,
such that the first entry of each $v_i$ is 1. (Hence the condition not
fulfiled is that $\pi_1(\Gamma')$ is one-to-one.) Furthermore, the
window $W$ should fit into $[0,1]^d$. Let us call a point set with all
these properties a \emph{canonical quasi-CPS}, 
cqCPS for short. Then we can study whether $W$ is a BRS by considering 
whether the corresponding cqCPS is bounded distance equivalent to some lattice.

In order to fit the window $W$ into $[0,1]^d$ the following two 
trivial observations are helpful. 

\begin{lem} \label{lem:lotsof-cps}
A CPS $\Lambda$ with lattice $\Gamma$, $G=\R$,
$H=\R^d$, and $W \subset H$ is the union of $k^{d+1}$ CPS
with lattice translates $k \Gamma + t$ $(t \in \Gamma / k \Gamma)$ 
and the same $G$, $H$, $W$.  
\end{lem}

\begin{lem} \label{lem:many-cps}
If $\Lambda_1, \ldots, \Lambda_{\ell}$ are disjoint CPS in $\R$ such that 
there is $a \in \R$ with $\Lambda_i \bd a \Z$ for $1 \le i \le \ell$, then 
$\Lambda_1 \dotcup \ldots \dotcup \Lambda_{\ell} \bd \frac{a}{\ell} \Z$.  
\end{lem}

A more interesting result on the converse of Lemma \ref{lem:many-cps} 
is Lemma \ref{lem:lattice_divide} in Section \ref{sec:bde-cps}.

The following connection between BRS and CPS is spelled out already in 
\cite{H}, see also references therein, e.g.\ for the special case $d=1$.

\begin{lem} \label{lem:bde-cqcps}
Let $\Lambda \subset \R$ be a CPS with window $W \subset \R^d$ and 
lattice $\Gamma$. Then $\Lambda$ is bounded distance equivalent 
to some lattice $a\Z$ if and only if the corresponding 
cqCPS $\Lambda'$ constructed above is bounded distance equivalent 
to some lattice $b\Z$.  
\end{lem}
Of course in this case $a = \dens^{-1}(\Lambda)$ and
$b = \dens^{-1}(\Lambda')$. 
\begin{proof}
Let us make the relation sketched above more precise. Let $\Lambda$ 
be some CPS with lattice $\Gamma \subset \R \times \R^d$ and window $W$.  
Choose a $d$-dimensional sublattice $\Gamma_d$ of $\Gamma$ that 
is maximal in the sense that $\Gamma_d = \Gamma \cap H'$, where $H'$ 
is some $d$-dimensional hyperplane. This step corresponds to 
\emph{choosing} the direction of projection: we want to project along 
some lattice vector (if $H=\R$), respectively along a $d$-dimensional
sublattice (if $H=\R^d$, $d \ge 1$).

By Lemmas \ref{lem:lotsof-cps} and \ref{lem:many-cps} we can assume
without loss of generality that there is a fundamental 
domain $F$ of $\Gamma_d$ whose projection $\pi_2(F)$ contains $W$ in 
its interior. (Otherwise we split the CPS $\Lambda$ into $k^{d+1}$ coarser 
ones using Lemma \ref{lem:lotsof-cps}. Showing that each of them is
bounded distance equivalent to some $k^{d+1} a \Z$ implies that 
$\Lambda \bd a \Z$.)

The next step is to \emph{change} the direction of projection, or 
alternatively: applying some linear map to the lattice $\Gamma$
(and keeping $\pi_1$ being the orthogonal projection to $G=\R$).
$H'$ can be written as $(l(y), y)^T$ ($y \in \R^d$) with $l: \R^d \to \R$ linear. 
Hence we apply the map $\phi: \, (x_1, x_2, \ldots, x_{d+1})^T
\mapsto (x_1-l(x_2, \ldots, x_{d+1}), x_2, \ldots, x_{d+1})^T$ 
to $\Gamma$. 
This transformation maps $H'$ to $H$ and hence $\Gamma_d$ to $H$. 
While $\phi$ is not a bounded distance transformation
with respect to the entire lattice $\Gamma$, it is one
with respect to $\Gamma \cap (\R \times W)$, since
$l$ is bounded on $W$. Hence the CPS $\Lambda$ will be
mapped to $\Lambda'$, where $\Lambda' \bd \Lambda$. 
Since $\Gamma':=\phi(\Gamma)$ consists of $d$-dimensional layers
of copies of $\pi_2(\Gamma_d)$
parallel to $H=\R^d$, $\Lambda'$ is already a subset of some 
periodic set $c \Z$. Without loss of generality let $c=1$.

By shifting $\Gamma'$ and $W$ along $H$ simultaneously we can 
achieve that the parallelepiped $\pi_2(F)$ has one vertex at 0. 
This shift leaves $\Lambda'$ invariant.
Transforming $\pi_2(F)$ into a unit cube $[0,1]^d$ 
(hence $\pi_2(\Gamma')$ into $\Z^d$) is achieved
by a linear map $\psi$ that does only affect the internal 
space $\R^d$, hence $\psi$ keeps $\Lambda'$ unchanged. Thus
$\Lambda'$ is a cqCPS.
\end{proof}

 
\begin{thm} \label{thm:brs-cps}
Let $\Lambda \subset \Z$ be a cqCPS with window $W \subset \R^d$, lattice 
$\Gamma$ and slope $\alpha$. Then $\Lambda$ is bounded distance 
equivalent to the lattice $\dens^{-1}(\Lambda) \Z$ if and only 
if $W$ is a BRS with respect to $\alpha$. 
\end{thm}
\begin{proof}
Without loss of generality, let $\Lambda = \{ \ldots, \lambda_{-1}, 
\lambda_0=0, \lambda_1, \lambda_2, \ldots \}$. (If $0 \notin \Lambda$
we may consider $\Lambda-k$ for some $k \in \Lambda$.) 

By definition, $W$ is a BRS with respect to $\alpha$ if and only if there 
is $C>0$ such that for all $n$
\begin{equation} \label{eq:defic}
|D_n(W,x)| = |\sum_{k=0}^{n-1} 1_W(x+k \alpha \bmod 1) \; - n \mu(W)|<C.
\end{equation}
Let $\lambda_m \in \Lambda$ and let $n_m-1$ be the summation index yielding
$\lambda_m$; that is, $n_m-1=\lambda_m$, and
\[ m= \sum\limits_{k=0}^{n_m-1} 1_W(x+k \alpha \bmod 1). \]
Using $O$ notation, \eqref{eq:defic} becomes $m = n_m \mu(W) + O(1)$.
This is equivalent to 
\[ \frac{1}{\mu(W)} m = n_m + O(1) = n_m -1 + O(1), \]
hence to $\frac{1}{\mu(W)} m = \lambda_m + O(1)$. This means that the
distance between the $m$th point $\lambda_m$ of $\Lambda$ and the
$m$th point of $\frac{1}{\mu(W)} \Z = \dens^{-1}(\Lambda) \Z$ is
uniformly bounded in $m$, hence $\Lambda \bd \dens^{-1}(\Lambda) \Z$.
\end{proof}

Note that the transformations in the proof of Lemma \ref{lem:bde-cqcps},
transforming a CPS into a cqCPS, changes the direction of the slope $\alpha$,
as well as the shape of the window $W$. 
Whenever we use Theorem \ref{thm:brs-cps} together with Lemma \ref{lem:bde-cqcps}
in order to show that the window $W$ of some CPS is a BRS one has to read 
this as ``an affine image of $W$ is a BRS with respect to an appropriate 
slope $\alpha$''.

\section{BRS with fractal boundary}

In the sequel we need a notion of equidecomposability. 
\begin{defi}\label{defi:decompose}
Let $X$ and $Y$ be two subsets of $\R^d$, and let $V$ be a set of 
$d$-dimensional vectors. We say that $X$ and $Y$ are {\it $V$-equidecomposable} 
if there is a natural number $k$, a decomposition of $X$ into $k$ disjoint subsets 
\[ X=X_1\dotcup \ldots \dotcup X_k, \]
and $k$ vectors $\mathbf{v}_1, \ldots,\mathbf{v}_1 \in V$ such that
\[ Y=(X_1+\mathbf{v}_1)\dotcup \ldots \dotcup (X_k+\mathbf{v}_k). \]
\end{defi}
There are several profound results yielding conditions that imply that some 
given compact set in $\R^d$ is a BRS. For $d=1$ see for instance Theorem 
\ref{thm:kes}. For $d>1$ there are several results for special cases, see 
\cite{GLev, HKK} and references therein. A pretty general result was obtained 
by Grepstad and Lev.
\begin{thm}[\cite{GLev}] \label{thm:grep-lev}
Let $\alpha \in \R^d$ be completely irrational. That is, $\alpha=
(\alpha_1, \ldots, \alpha_d)$ such that $1,\alpha_1, \ldots, \alpha_d$ 
are linearly independent over $\Q$.
Then any zonotope in $\R^d$ with vertices belonging to $\Z^d + \alpha \Z$
is a BRS.\\
A Riemann measurable set $S \in \R^d$ is a BRS if and only if $S$ is 
($\Z^d + \alpha \Z$)-equidecomposable to some parallelepiped spanned 
by vectors in $\Z^d + \alpha \Z$.
\end{thm}

The connection between BRS and cqCPS established in Theorem \ref{thm:brs-cps}
allows us to use Theorem \ref{thm:grep-lev} to construct several cqCPS
(or CPS) that are bounded distance equivalent to some lattice $a \Z$. 
But it allows us also to go the opposite direction: the window of any CPS 
$\Lambda \subset \R$ with $\Lambda \bd a \Z$ yields a BRS. Candidates for
such CPS are provided by Theorem \ref{thm:pisot-bde}: discrete
point sets $\Lambda$ generated by Pisot substitutions \emph{are} bounded
distance equivalent to some $a \Z$. Hence the Pisot conjecture (or Pisot
substitution conjecture) becomes important. In fact, the Pisot conjecture
appears in several equivalent formulations, and in several levels of
generalization \cite{ABBLS}. The formulation needed here is the following.
  
\begin{con} \label{con:pisot} 
Let $\sigma$ be a primitive substitution on $\mathcal{A}=\{1, \ldots, n\}$,
If $\sigma$ is a Pisot substitution, and the characteristic polynomial of
$M_{\sigma}$ is irreducible over $\Q$, then the vertex sets of the tilings 
defined by $\sigma$ (in the sense of Section \ref{sec:subst-cps}) are CPS. 
\end{con}

In the case of two letters the conjecture is known to be true, as 
shown by Hollander and Solomyak \cite{HS}.

\begin{thm} \label{thm:pisot2letters}
Let $\sigma$ be a primitive substitution on $\mathcal{A}=\{1, 2 \}$. 
If $\sigma$ is a Pisot substitution, and the characteristic polynomial 
of $M_{\sigma}$ is irreducible, then the tilings defined by $\sigma$ 
(in the sense of Section \ref{sec:subst-cps}) are CPS. 
\end{thm}

As a consequence of Theorems  \ref{thm:pisot-bde}, \ref{thm:brs-cps}, and 
\ref{thm:pisot2letters}, the window of the CPS of any Pisot substitution
on two letters is a BRS, if the inflation factor is some quadratic irrational.
Applied to the Fibonacci sequence this yields
nothing new; compare Example \ref{ex:fib2}: the window is an interval
with length $\tau$, hence by Kesten's theorem this is a BRS. More 
interestingly, there are examples where the window has a boundary
that is a fractal. Consider for instance the substitution 
$a \mapsto aab$, $b \mapsto ba$. The substitution matrix is
$\big( \begin{smallmatrix} 2 & 1\\ 1 & 1 \end{smallmatrix} \big)$, the 
eigenvalues are $\tau^2$ and $\tau^{-2}$. In particular, the inflation factor
$\tau^2$ is a Pisot number, and $\sigma$ is a Pisot substitution.
Hence by Theorem \ref{thm:pisot-bde} for the vertex set $\Lambda$ of 
the tilings of the line generated by $\sigma$ holds: $\Lambda \bd a \Z$.
By Theorem \ref{thm:pisot2letters} the tilings of the line generated by 
$\sigma$ are CPS. Hence by Theorem \ref{thm:brs-cps} the window of
this CPS is a BRS. The window for $\Lambda$ is indicated in Figure 
\ref{fig:windows} (left). 
\begin{figure}
\includegraphics[angle=90,width=.05\textwidth]{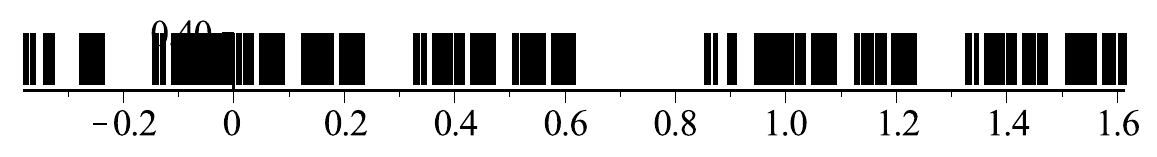}
\includegraphics[width=.4\textwidth]{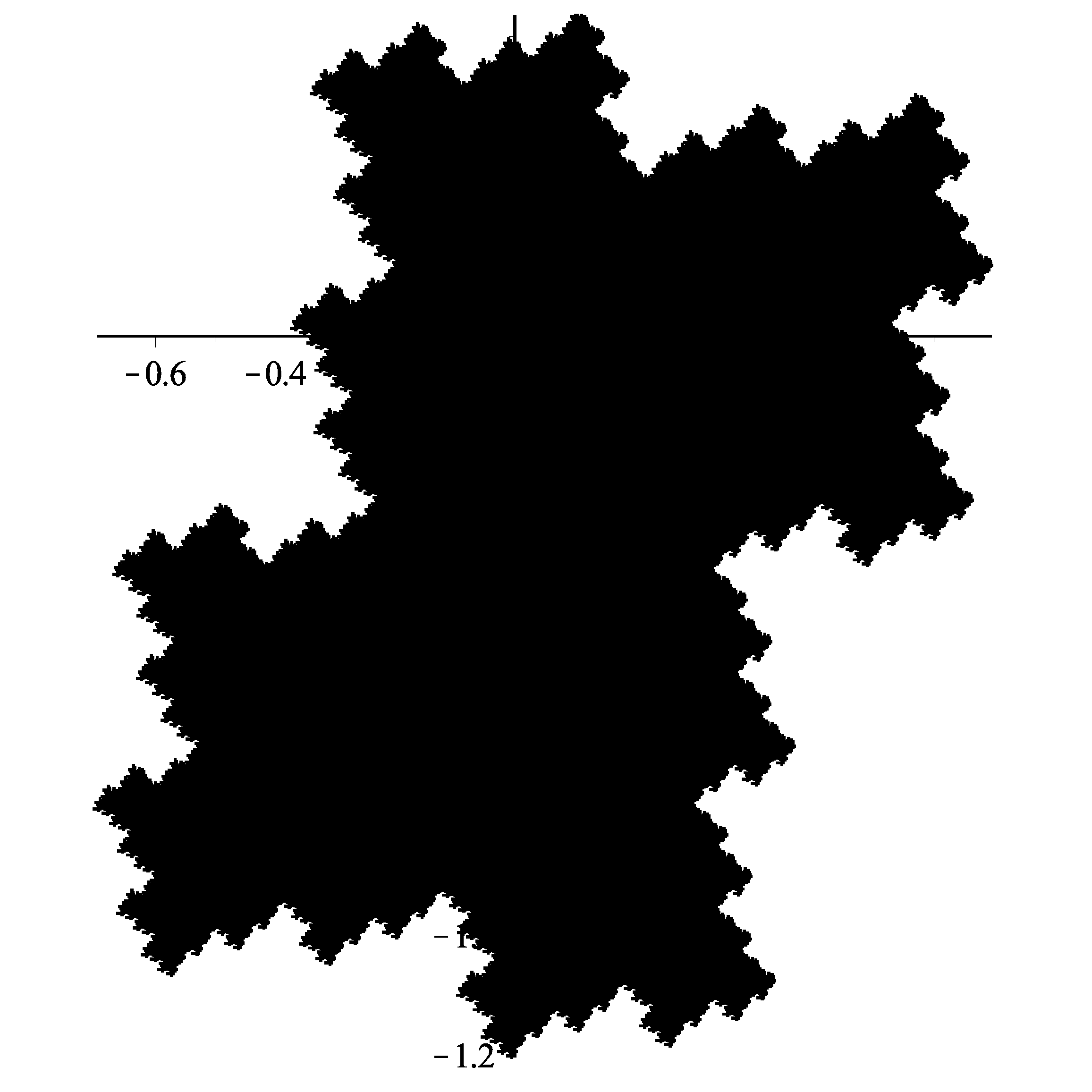} 
\includegraphics[width=.4\textwidth]{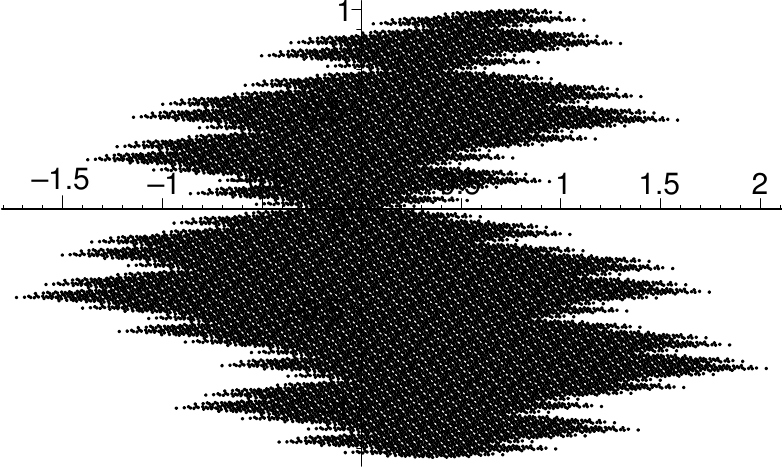}
\caption{The window for $a \mapsto aab$, $b \mapsto ba$ (left) is
one-dimensional. Here it is thickened for the sake of visualization.
The windows for $a \to abc$, $b \mapsto ab$, $c \mapsto b$ (middle)
and for $a \to abc$, $b \mapsto ab$, $c \mapsto a$ 
(right) are both BRS.
\label{fig:windows}}
\end{figure}
This window is similar to a (fat) Cantor set: it has infinitely many
connected components. Its total length is $\tau$, hence by Theorem
\ref{thm:grep-lev} it is not only a BRS, but equidecomposable to
the interval $[0,\tau]$. The window $W$ is the (measurewise disjoint) 
union of a coupled iterated function system with two sets $A$ and $B$:
\[ A = (\tau^{-2} A) \cup (\tau^{-2} A+1) \cup (\tau^{-2} B+2), \quad
B = (\tau^{-2} B) \cup (\tau^{-2} A+2) \]
By a theorem of Hutchinson \cite{Hut} such an iterated function system has
a unique nonempty compact solution $(A,B)$. Here we have $W=A \cup B$. 

In order to obtain higher dimensional BRS one can use Pisot 
substitutions with $d>2$ letters. Generically this requires
$H=\R^{d-1}$, yielding $(d-1)$-dimensional windows that are candidates of BRS.
Unfortunately we can not apply an analogue of Theorem 
\ref{thm:pisot2letters}: up to the knowledge of the authors the 
Pisot conjecture is still open for more than two letters.
The fractal nature of many window sets for substitutions
with more than two letters can make it tedious to prove that a given
Pisot substitution actually yields a CPS. In \cite{BS} it is
shown (on several pages) that the sequence generated by 
the Pisot substitution $a \to abc$, $b \mapsto ab$, $c \mapsto b$ 
indeed yields a CPS. Hence the window $W$ for this CPS is a BRS.
The slope for this example is $(\mbox{Re}(\beta), \mbox{Im}(\beta))^T$, 
where $\beta$ is the complex (non-real) root of $x^3-2x^2-1$ with positive 
imaginary part. This set is shown in Figure \ref{fig:windows} (middle).
The Hausdorff dimension of its boundary is $1.2167\ldots$ \cite{BS}.
The corresponding iterated function system is given in \cite{BS} or 
\cite[Example 7.5]{BG}. 

In such a way one gets a wealth of examples of BRS with fractal boundary. 
In \cite{AGL} the authors proved that all three-letter Pisot substitutions
where the trace of the substitution matrix is less than three
fulfil the Pisot conjecture. They checked all 446683 Pisot substitutions 
with this properties using a computer and the so-called potential
overlap algorithm \cite{AL}. One BRS arising from of these 
examples is shown in Figure \ref{fig:windows} (right), obtained 
from the Pisot substitution $a \to abc$, $b \mapsto ab$, $c \mapsto a$; 
hence with inflation matrix $\Big( \begin{smallmatrix} 1 & 1 & 1
\\ 1 & 1 & 0 \\ 1 & 0 & 0 \end{smallmatrix} \Big)$ \cite{F}. 
It is a BRS with respect to the slope $\alpha'=(\alpha_1, \alpha_2)$,
where $\alpha_1 = -0.8019\ldots$ and $\alpha_2=0.5549\ldots$ are the 
smaller roots of $x^3-2x^2-x+1$.

By Theorem \ref{thm:grep-lev} the last two examples are 
$(\Z^2 + (\mbox{Re}(\beta), \mbox{Im}(\beta))^T \Z)$-equidecomposable 
(respectively $(\Z^2+\alpha'\Z)$-equidecomposable) to parallelograms
spanned by vectors in $\Z^2 + (\mbox{Re}(\beta), \mbox{Im}(\beta))^T \Z$
(respectively in $\Z^2+\alpha'\Z$). See also \cite[Figure 7.3]{BG} or 
\cite[Figure 2.2]{ST} for a particularly fuzzy example of a set that 
actually is a BRS. 

\section{Bounded distance equivalence for two cut-and-project sets}
\label{sec:bde-cps}

Jamie Walton \cite{Wal} asked whether two CPS in the same hull are necessarily
bounded distance equivalent. In this section we will show that there
are two one-dimensional CPS in the same hull that are not bounded distance 
equivalent. Recall that the hull was defined for primitive substitutions
above. More general, the hull of a discrete point set $\Lambda \subset \R$
is defined as the closure of $\{ t + \Lambda \mid t \in \R \}$ under the
local topology, see \cite[Def.~4.7]{BG} for details. Roughly speaking, being in the 
same hull is the precise concept of $\Lambda$ and $\Lambda'$ being ``locally indistinguishable'', 
i.e., $\Lambda$ contains a translate of each finite pattern in $\Lambda'$
and vice versa. This is literally true only if $\Lambda$ and $\Lambda'$ 
fulfil certain nice properties, see \cite[Section 4.2]{BG}. 

For CPS this essentially boils down to the statement that
two CPS are in the same hull if they can be generated with the
same $G,H,\Gamma$, and the windows are translates of each other. (There
are some subtleties to consider, but this is literally true for $G=\R=H$ 
and the windows being half-open intervals.) So the hull of a CPS can
be obtained by considering all translations of the window $W$.

Theorem \ref{thm:win-bde} and Lemmas \ref{lem:hall}
and \ref{lem:lattice_divide} below are valid also in higher dimensions.
For this purpose note that CPS can live as well in $\R^e$ for some $e>1$:
just replace in Definition \ref{def:cps} of a CPS above $G=\R$ by $G=\R^e$.
  
\begin{thm} \label{thm:win-bde}
Let $\Lambda$ and $\Lambda'$ be two CPS defined with the same lattice $\Gamma$ and with the same projections $\pi_1$ and $\pi_2$ but with different windows $W$ and $W'$ respectively.

If the windows $W$ and $W'$ are $\pi_2(\Gamma)$-equidecomposable, then $\Lambda \bd\Lambda'$.
\end{thm}
\begin{proof}
Let $W=W_1\dotcup \ldots \dotcup W_k$ be the decomposition given in Definition \ref{defi:decompose}, and let $\mathbf{x}_i\in \Gamma$ be the vectors such that 
$$W'=(W_1+\pi_2(\mathbf{x}_1))\dotcup \ldots \dotcup (W_k+\pi_2(\mathbf{x}_k)).$$
In this case we can represent $\Lambda$ as disjoint union $\Lambda=\Lambda_1\dotcup \ldots \dotcup \Lambda_k$ where 
$$\Lambda_i=\{\pi_1(x) \mid x\in \Gamma, \pi_2(x)\in W_i\}.$$

Therefore the set 
\begin{multline*}\Lambda_i+\pi_1(\mathbf{x}_i)=\{\pi_1(x+\mathbf{x}_i) \mid x\in \Gamma, \pi_2(x)\in W_i\}=\\=\{\pi_1(x+\mathbf{x}_i) \mid x\in \Gamma, \pi_2(x+\mathbf{x}_i)\in W_i+\pi_2(\mathbf{x}_i)\}=\{\pi_1(y) \mid y\in \Gamma, \pi_2(y)\in W_i+\pi_2(\mathbf{x}_i)\}
\end{multline*}
is a subset of $\Lambda'$. Moreover, since the sets $W_i+\pi_2(\mathbf{x}_i)$ decompose $W'$ we can decompose $\Lambda'$ into
$$\Lambda'=(\Lambda_1+\pi_1(\mathbf{x}_1))\dotcup \ldots \dotcup (\Lambda_k+\pi_1(\mathbf{x}_1)).$$

We define $f:\Lambda\longrightarrow \Lambda'$ by $f(x):= x+\pi_1(\mathbf{x}_i)$ provided $x\in \Lambda_i$. This map is clearly a bijection and $||f(x)-x||$ is bounded by the length of the longest vector from $\mathbf{x}_i$.
\end{proof}

In particular two $\R^d\times\R$ CPS with windows $W$ and $W'$ differ by a translation of a vector from $\pi_2(\Gamma)$ are bounded distance equivalent (actually they are translated copies since $k=1$ in that case). Our next goal is to show that this condition can't be relaxed. Namely we will show that two halves of the Fibonacci sequence are not bounded distance equivalent although they can be constructed as $\R\times \R$ CPS with windows being intervals of equal length.

We are going to use the following theorem due to Rado \cite{Rado}. It is an infinite version of
Hall's marriage lemma.

\begin{lem}\label{lem:hall}
Consider a bipartite graph $G=(V_1 \cup V_2, E)$ with countable vertex sets 
$V_1$ and $V_2$ such that any vertex has finite degree. Assume that for any finite subset $X$ of $V_1$ there are at least $|X|$ vertices from $V_2$ that are connected with at least one vertex from $X$; and the same holds for the roles of $V_1$ and $V_2$
exchanged.  Then there exists a set of disjoint edges of $G$ that covers all vertices of $G$.
\end{lem}

\begin{lem}\label{lem:lattice_divide}
Consider a discrete point set set $A$ in $\R^d$ such that $A\bd\mathbb{Z}^d.$ If $A$ is represented as $A=M_1\dotcup M_2\dotcup \ldots \dotcup M_n$ in such a way that $M_i\bd M_j$ then $M_i\bd n^{1/d} \mathbb{Z}^d.$
\end{lem}
The statement can be easily deduced from Laczkovich's criterion \cite{Lac} but we would like to provide a purely combinatorial proof below.

\begin{proof}
Without loss of generality we can assume that $A=\mathbb{Z}^d.$ It is enough to construct a bounded distance bijection between $M_1$ and the set of all integer points in $T:=\{(nk_1,k_2,\ldots,k_n) \, | \, k_i \in \Z\}$. Two lattices of the same density are bounded distance equivalent \cite[Theorem 1]{DO1}, so we have $n^{1/d} \mathbb{Z}^d \bd T$ and the claim of the lemma immediately follows from them. Let $c_{ij}$ be a distance in the bounded distance bijection between $M_i$ and $M_j$ and $c=\max c_{ij}$. 

Consider a bipartite graph with parts represented by points from $M_1$ and points from $T.$ Two vertices are connected with an edge if and only if the distance between these vertices is at most $c+n.$ We will show that this graph satisfies Lemma \ref{lem:hall}. For every $k$ and for any $k$ points from $M_1$ we have at least $kn$ lattice points in $A$ covered by $c$-balls centred at these points (we have at least $k$ points from each $M_i$). For every such lattice point $x$ we decrease the first coordinate to the closest integer divisible by $n$. In this way we will get a point $y\in T$ with distance at most $n$ from $x$. The point $y$ is at distance at most $c+n$ from one of the $k$ points from $M_1$ we used initially. For a fixed point $y\in T$ we reached $y$ at most $n$ times so we obtained at least $k$ different points from $T$ in $(c+n)$-balls centred at given $k$ points of $M_1$.

The same is true for any $k$ points from $T$: there are $kn$ points in this $n\times 1\times\ldots\times 1$ bricks and at least $k$ of them are from one $M_i$. Thus at least $k$ points from $M_1$ are at distance $c$ from these bricks and at distance $c+n$ from these points from $T.$ So the conditions of Lemma \ref{lem:hall} are fulfilled and we get the desired bounded distance bijection determined by the edges from Lemma \ref{lem:hall}.
\end{proof}

Recall that the Fibonacci sequence $\Lambda_F$ is the CPS with $d=1$, $W=[-\frac{1}{\tau}, 1[$, lattice $\Gamma = \langle \big(\begin{smallmatrix} 1\\ 1 \end{smallmatrix} \big), \big( 
\begin{smallmatrix} \tau \\ \frac{-1}{\tau} \end{smallmatrix} \big) \rangle_{\Z}$, 
and $\pi_1$ and $\pi_2$ are orthogonal projections.

Let $F_1$ and $F_2$ be two CPS with the same initial data as the Fibonacci sequence but with $W_1=[-\frac{1}{\tau}, \frac{1-\frac{1}{\tau}}{2}[$ and $W_2=[\frac{1-\frac{1}{\tau}}{2}, 1[$;  that is, $F_1$ and $F_2$ use disjoint halves of the window $W$. Therefore we call them {\it half-Fibonacci sequence(s)}. Note, that $F_1$ and $F_2$ use windows of the same length, so $F_1$ and $F_2$ are contained in the same hull, but these windows are not translated to each other by a vector from $\pi_2(\Gamma)$. 

\begin{thm} \label{thm:not-bde}
The Half-Fibonacci sequences $F_1$ and $F_2$ are not bounded distance equivalent.
\end{thm}
\begin{proof}
Assume $F_1\bd F_2$. The union $F=F_1\dotcup F_2$ is the Fibonacci set 
and according to Kesten's theorem \ref{thm:kes} is bounded distance 
equivalent to the lattice $\alpha\Z$ for $\alpha=\dens^{-1}(F)$. Then 
according to Lemma \ref{lem:lattice_divide} 
$F_1\bd F_2\bd 2\alpha\Z$.

However, according to Kesten's theorem \ref{thm:kes} $F_1$ is not bounded distance equivalent to a lattice since the length of the window does not equal to the length of $\pi_2$-projection of any vector of $\Gamma$. We've got a contradiction and hence $F_1$ and $F_2$ are not bounded distance equivalent.
\end{proof}

\section{Outlook}

The results of this paper raise several further questions. We would like
to point out a few of them.

\begin{enumerate}
\item In the context of Theorem \ref{thm:not-bde} one may ask: Given 
some discrete point set (or tiling), generated by substitution or by some 
cut-and-project method, of how many equivalence classes (with
respect to $\bd$) does its hull consist? Theorem \ref{thm:hull-lattice}
yields a partial answer to this question for substitution sequences: 
if for any $\Lambda \in \X_{\sigma}$ holds $\Lambda \bd a \Z$, then 
for all $\Lambda' \in \X_{\sigma}$ holds $\Lambda' \bd a \Z$, hence
there is only one equivalence class. In particular one may ask:
In the hull of some CPS, or some primitive substitution, is the number of  
equivalence classes with respect to $\bd$ always either one or infinite?
\item It would also be interesting to expand on Remark \ref{rem:aabbab}:
For any primitive substitution $\sigma$ with inflation factor $\lambda \in \Z$, 
can there be more than one equivalence class in the hull $\X_{\sigma}$?
Respectively, can there be a substitution sequence $\Lambda$ in
$\X_{\sigma}$ that is not bounded distance equivalent to any $a \Z$?
\item On a different note it might be worth to study whether the results of
this paper yield a BRS beyond the scope of Theorem \ref{thm:grep-lev}.
For instance it might be worth studying whether this construction 
yields some BRS with respect to a slope that is not completely irrational.
In all examples we tried the required affine maps always leave
$\Z^d + \alpha \Z$ invariant. Hence these examples yielded nothing 
beyond the scope of Theorem \ref{thm:grep-lev}. So we ask: is there
a CPS that yields a BRS not covered by Theorem \ref{thm:grep-lev}?
\end{enumerate}

\section*{Acknowledgements}
We are grateful to Franz G\"ahler and Jamie Walton for fruitful discussions,
in particular Section \ref{sec:cps-brs} (Franz) and Section \ref{sec:bde-cps}
(Jamie). We are also thankful to the anonymous referee whose remarks improved
the quality of the paper. DF acknowledges support by the Research Center for 
Mathematical Modelling of Bielefeld University.

\end{document}